\documentclass[12pt,reqno]{amsart}
\usepackage{mathrsfs,amssymb,graphicx,verbatim,amsmath,bm}
\usepackage{paralist}
\usepackage[breaklinks,pdfstartview=FitH]{hyperref}
\renewcommand{\d}{\delta}

\newcommand{\e}{\varepsilon}
\newcommand{\R}{\mathbb R}

\newcommand{\eqdef}{\stackrel{\mathrm{def}}{=}}

\renewcommand{\le}{\leqslant}
\renewcommand{\ge}{\geqslant}
\renewcommand{\leq}{\leqslant}

\newtheorem{theorem}{Theorem}[section]

\newtheorem{lemma}[theorem]{Lemma}
\newtheorem{corollary}[theorem]{Corollary}

\theoremstyle{remark}

\newtheorem{remark}[theorem]{Remark}

\theoremstyle{definition}

\renewcommand{\subset}{\subseteq}
\renewcommand{\supset}{\supseteq}

\newcommand{\C}{\mathbb C}

\renewcommand{\H}{\mathbb H}

\newcommand{\N}{\mathbb N}

\renewcommand{\setminus}{\smallsetminus}

\newcommand{\f}{\phi}

\begin{document}

\title[Infinite dimensional doubling subset of $L_p$]{
A doubling subset of $L_p$ for $p>2$ that is inherently
infinite dimensional}

\author{Vincent Lafforgue}
\address{Laboratoire de Math\'ematiques - Analyse, Probabilit\'es, Mod\'elisation - Orl\'eans (MAPMO)
UMR CNRS 6628, Universit\'e d'Orl\'eans
Rue de Chartres, B.P. 6759 - 45067 Orl\'eans cedex 2} \email{vlafforg@math.jussieu.fr}

\author{Assaf Naor}
\address{Courant Institute of Mathematical Sciences, New York University, 251 Mercer Street, New York NY 10012, USA}
\email{naor@cims.nyu.edu}

\begin{abstract}
It is shown that for every $p\in (2,\infty)$ there exists a doubling
subset of $L_p$ that does not admit a bi-Lipschitz embedding into
$\R^k$ for any $k\in \N$.
\end{abstract}

\maketitle

\section{Introduction}

Given $K\in [1,\infty)$, a metric space $X$ is said to be
$K$-doubling if every ball in $X$ can be covered by at most
$K$ balls of half its radius. $X$  is said to be doubling if it is
$K$-doubling for some $K\in (0,\infty)$.

Lang and Plaut~\cite{LP01} asked whether or not every doubling
subset of Hilbert space admits a bi-Lipschitz embedding into $\R^k$
for some $k\in \N$. We refer to~\cite[Sec.~1.1]{NN12} for further
discussion on the ramifications of this question, as well as a
construction of a doubling subset of Hilbert space that conceivably does not to admit a bi-Lipschitz embedding into $\R^k$ for
any $k\in \N$. While the validity of this suggestion of~\cite{NN12}
remains open, and hence also the Lang--Plaut question remains open,
here we show that a variant of the method that was proposed
in~\cite{NN12} can be used to prove that the analogue of the
Lang--Plaut problem with Hilbert space replaced by $L_p$ for $p\in
(2,\infty)$ has a negative answer.

\begin{theorem}\label{thm:main1}
For every $p\in (2,\infty)$ there exists a doubling subset of $L_p$
that does not admit a bi-Lipschitz embedding into $\R^k$ for any
$k\in \N$.
\end{theorem}

We thank Ofer Neiman for asking us the question that
Theorem~\ref{thm:main1} answers. In~\cite{BGN13} Bartal, Gottlieb and Neiman concurrently found a
construction that also yields Theorem~\ref{thm:main1}; their
(combinatorial) construction is entirely different from our
(analytic) construction. The potential validity of Theorem~\ref{thm:main1} for $p\in (1,2]$ remains an open question, while for $p=1$ stronger results are known; see Remark~\ref{rem:finite dim} below.

Theorem~\ref{thm:main1} is a special case of the following result.

\begin{theorem}\label{thm:main2}
For every $p\in (2,\infty)$ there exists a doubling subset $D_p$ of
$L_p$ that does not admit a bi-Lipschitz embedding into $L_q$ for
any $q\in (1,p)$. Furthermore, there exists $p_0\in (2,\infty)$ such
that  $D_p$ does not admit a bi-Lipschitz embedding into $L_1$ for
every $p\in [p_0,\infty)$.
\end{theorem}

Theorem~\ref{thm:main2} is a formal consequence of the following
finitary result.

\begin{theorem}\label{thm:main 3}
There exists a universal constant $K\in [1,\infty)$ and for every
$n\in \N$ there exists an $n$-point metric space $(X_n,d_{X_n})$
with the following properties. For every $n\in \N$ and $p\in
[2,\infty)$ there exists a mapping $f_{n,p}:X_n\to L_p$ such that
$f_{n,p}(X_n)\subset L_p$ is $K$-doubling, and
$$
\forall\, x,y\in X_n,\quad d_{X_n}(x,y)\le \|f_{n,p}(x)-f_{n,p}(y)\|_p\lesssim (\log n)^{1/p}d_{X_n}(x,y).
$$
Moreover, for every $q\in (1,\infty)$, any embedding of $X_n$ into
$L_q$ incurs distortion at least $c(q)(\log n)^{\min\{1/2,1/q\}}$,
where $c(q)\in (0,\infty)$ may depend only on $q$. Any embedding of
$X_n$ into $L_1$ incurs distortion at least $(\log n)^c$ for some
universal constant $c\in (0,1/2]$.
\end{theorem}

Here and in what follows, the notations $A\lesssim B$ and $B\gtrsim
A$ mean that $A\le CB$ for some universal constant $C\in
(0,\infty)$. If we need to allow $C$ to depend on parameters, we
indicate this by subscripts, thus e.g. $A \lesssim_{\beta} B$
means that $A \leq C(\beta) B$ for some $C(\beta)\in
(0,\infty)$ which is allowed to depend only on the parameter
$\beta$. The notation $A\asymp B$ stands for $(A\lesssim
B) \wedge  (B\lesssim A)$.

The fact that Theorem~\ref{thm:main 3} implies
Theorem~\ref{thm:main2} is simple. Indeed, fix $p\in (2,\infty)$. By
a standard ``disjoint union" argument (see e.g. the beginning of
Section 4 in~\cite{NN12}), there exists a doubling subset $D_p$ of
$L_p$ that contains an isometric copy of a rescaling of
$f_{n,p}(X_n)$ for every $n\in \N$. If $q\in (1,p)$ and
$f_{n,p}(X_n)$ embeds with bi-Lipschitz distortion $M\in [1,\infty)$
into $L_q$, then by Theorem~\ref{thm:main 3} we have
$$
M\gtrsim_q (\log n)^{\min\left\{\frac12,\frac1{q}\right\}-\frac{1}{p}}\xrightarrow[n\to \infty]{} \infty.
$$
Consequently, $D_p$ does not admit a bi-Lipschitz embedding into
$L_q$. For $q=1$ the same argument shows that $D_p$ does not admit a
bi-Lipschitz embedding into $L_1$ provided $p>1/c$, where $c$ is the
(universal) constant from Theorem~\ref{thm:main 3}.

\begin{remark}\label{rem:finite dim}
The above reasoning implies that for every $p\in (2,\infty)$ and
every $n\in \N$ there exists an $n$-point $O(1)$-doubling subset
$S_n$ of $L_p$ such that for every $k\in \N$, if $S_n$ embeds with
distortion $M\in [1,\infty)$ into $\ell_p^k$ then necessarily,
$$
M\gtrsim \left(\frac{\log n}{k}\right)^{\frac1{2}-\frac{1}{p}}.
$$
This is true because $\ell_p^k$ embeds into Hilbert space with distortion $k^{\frac{1}{2}-\frac{1}{p}}$. It is open whether or not a similar statement holds true for $p\in
(1,2]$. For $p=1$ an even stronger lower bound  was shown to hold
true in~\cite{LMN05}: for every $n\in \N$ there exists an $n$-point
$O(1)$-doubling subset $A_n$ of $L_1$ such that for every $k\in \N$,
if $A_n$ embeds with distortion $M\in [1,\infty)$ into $\ell_1^k$
then necessarily

\begin{equation}\label{eq:laakso}
M\gtrsim \sqrt{\frac{\log n}{\log k}}.
\end{equation}
The examples leading to~\eqref{eq:laakso} are the Laakso graphs~\cite{Laakso-gafa}, which are doubling metric spaces that were shown to embed into $L_1$ in~\cite{GNRS04}. They yield a doubling subset of $L_1$ that does not admit a bi-Lipschitz embedding into Hilbert space~\cite{Laakso-gafa} (see also~\cite[Thm.~2.3]{LP01}), any uniformly convex Banach space~\cite{MN08-socg}, or even any Banach space with the Radon--Nikod\'ym property~\cite{Ost11}.
\end{remark}

\subsection{The Heisenberg group} The metric spaces $\{X_n\}_{n=1}^\infty$ of
Theorem~\ref{thm:main 3} arise from the discrete Heisenberg group.
To explain this, recall that the discrete Heisenberg group, denoted
$\H$, is the group generated by two elements $a,b\in \H$, with the
relations asserting that the commutator $[a,b]=aba^{-1}b^{-1}$ is in
the center of $\H$. Let $e_\H$ denote the identity element of $\H$.
The left-invariant word metric on $\H$ induced by the symmetric
generating set $\{a,b,a^{-1},b^{-1}\}$ is denoted
$d_W(\cdot,\cdot)$. For $r\in [1,\infty)$ let $B(r)= \{x\in \H:\
d_W(x,e_\H)\le r\}$ denote the corresponding closed ball of radius
$r$. Then $|B(r)|\asymp r^4$ (see e.g.~\cite{Bla03}). It follows
that there exists $\eta_1,\eta_2\in (0,\infty)$ such that for every
large enough $n\in \N$ there exists $X_n\subset \H$ with
$|X_n|=n$ and
\begin{equation}\label{eq:Xn inclusions}
B\left(\eta_1\sqrt[4]{n}\right)\subset X_n\subset B\left(\eta_2\sqrt[4]{n}\right).
\end{equation}

By virtue of the leftmost inclusion in~\eqref{eq:Xn inclusions}, the
distortion lower bounds that are asserted in Theorem~\ref{thm:main
3} follow from~\cite{ckn} for $q=1$, from~\cite{ANT10} for $q=2$ and
from~\cite{LN12} for $q\in (0,\infty)\setminus\{2\}$. The remaining
assertions of Theorem~\ref{thm:main 3} follow from
Theorem~\ref{thm:main4} below.

\begin{theorem}\label{thm:main4}
For every $\e\in (0,1/2]$ and $p\in [2,\infty)$ there exists a
mapping $F_{\e,p}:\H\to L_p$ such that $F_{\e,p}(\H)\subseteq L_p$
is $2^{16}$-doubling and
$$
\forall\, x,y\in \H,\qquad d_W(x,y)^{1-\e}\le \|F_{\e,p}(x)-F_{\e,p}(y)\|_p\lesssim \frac{d_W(x,y)^{1-\e}}{\e^{1/p}}.
$$
\end{theorem}
The case $p=2$ of Theorem~\ref{thm:main4} was previously proven
in~\cite{LN06} relying on Hilbertian arguments, namely on
Schoenberg's characterization~\cite{Sch38} of subsets of
Hilbert space through positive definite kernels. Here we find
a different approach that works also when $p\in (2,\infty)$. Note that \cite{LN06} contains a stronger statement that is used crucially in the context of~\cite{LN06} and does not follow from our proof.

Theorem~\ref{thm:main4} implies Theorem~\ref{thm:main 3} because in
light of the rightmost inclusion in~\eqref{eq:Xn inclusions}, all
the nonzero distances in $X_n$ are between $1$ and
$2\eta_2\sqrt[4]{n}$. Consequently, for $\e_n=1/\log n$ we have
$d_W(x,y)^{1-\e_n}\asymp d_W(x,y)$ for every $x,y\in X_n$. We can
therefore take $f_{n,p}=F_{\e_n,p}$ in Theorem~\ref{thm:main 3}.

\begin{remark}
The dependence on $\e$ in Theorem~\ref{thm:main4} is asymptotically
sharp as $\e\to 0$. For $p=2$ this was proven in~\cite{NN12} using
an inequality of~\cite{ANT10}. An analogous argument works for $p\in
(2,\infty)$ using~\cite{LN12} instead of~\cite{ANT10}. Indeed, write
$c=[a,b]$ (recall that $a,b$ are the generators of $\H$).
By~\cite{LN12} every $f:\H\to L_p$ satisfies

\begin{multline}\label{eq:LN quote}
\sum_{k=1}^{n^2}\sum_{x\in B_n}\frac{
\|f(xc^k)-f(x)\|_p^p}{k^{1+p/2}}\\\lesssim_p \sum_{x\in B_{21n}}
\Big(\|f(xa)-f(x)\|^p_p+\|f(xb)-f(x)\|^p_p\Big).
\end{multline}
Suppose that $M\in [1,\infty)$ satisfies
\begin{equation}\label{eq:M contrapositive}
\forall\, x,y\in \H,\qquad d_W(x,y)^{1-\e}\le \|f(x)-f(y)\|_p\le Md_W(x,y)^{1-\e}.
\end{equation}
Since $d_W(c^k,e_\H)\asymp \sqrt{k}$ and $|B(m)|\asymp m^4$ for every $k,m\in \N$ (see
e.g.~\cite{Bla03}), by substituting~\eqref{eq:M contrapositive}
into~\eqref{eq:LN quote} we see that
$$
\sum_{k=1}^{n^2}n^4\cdot \frac{k^{(1-\e)p/2}}{k^{1+p/2}}\lesssim_p M^pn^4.
$$
Hence $M^p\gtrsim_p \sum_{k=1}^\infty\frac{1}{k^{1+\e p/2}}\gtrsim_p
\frac{1}{\e}$, so $M\gtrsim_p 1/\e^{1/p}$.
\end{remark}

\section{Proof of Theorem~\ref{thm:main4}}

For every $n\in \N$ and $x\in \R^{2n+1}$ let $\pi(x)\in \R^{2n}$ denote the canonical projection of $x$ to $\R^{2n}$, i.e,
$$
\pi(x_1,x_2,\ldots,x_{2n},x_{2n+1})\eqdef (x_1,\ldots,x_{2n}).
$$
For $x,y\in \R^{2n}$ write
$$
\langle x,y\rangle\eqdef \sum_{j=1}^{2n}x_jy_j\qquad\mathrm{and} \qquad [x,y]\eqdef \sum_{j=1}^n (x_{2j-1}y_{2j}-x_{2j}y_{2j-1}).
$$
We also write as usual $\|x\|_2\eqdef \sqrt{\langle x,x\rangle}$.

The Heisenberg group product on $\R^{2n+1}$ is defined as follows. For every $x,y\in \R^{2n+1}$ let $xy\in \R^{2n+1}\cong \R^{2n}\times \R$  be defined as
\begin{equation*}\label{eq:def Heisenberg product}
xy\eqdef \left(\pi(x)+\pi(y),x_{2n+1}+y_{2n+1}-2[\pi(x),\pi(y)]\right).
\end{equation*}
Under this product $\R^{2n+1}$ becomes a noncommutative group,
called the $n$th (continuous) Heisenberg group and denoted $\H_n$,
whose identity element is $0\in \R^{2n+1}$ and the multiplicative
inverse of $x\in \R^{2n+1}$ is given by $x^{-1}=-x$; see
e.g.~\cite{Sem03}. The Lebesgue measure on $\R^{2n+1}$ is a Haar
measure for $\H_n$.

The Kor\'anyi norm of $x\in \R^{2n+1}$ is defined by
$$
N(x)\eqdef \left(\|\pi(x)\|_2^4+x_{2n+1}^2\right)^{1/4}.
$$
As shown in e.g.~\cite{Cyg81}, we have $N(xy)\le N(x)+N(y)$ for every $x,y\in \R^{2n+1}$. Consequently, if we set
$$
\forall\, x,y\in \R^{2n+1},\qquad d_N(x,y)\eqdef N(x^{-1}y),
$$
then $d_N$ is a left-invariant metric on $\H_n$. For every $\theta\in \R$ define $\d_\theta:\R^{2n+1}\to \R^{2n+1}$ by
\begin{equation}\label{eq:def delta theta}
\forall\, x\in \R^{2n+1},\qquad \d_\theta(x)\eqdef \left(\theta \pi(x),\theta^2x_{2n+1}\right).
\end{equation}
Then $d_N(\d_\theta(x),\d_\theta(y))=|\theta|d_N(x,y)$ for every $x,y\in \R^{2n+1}$.

Fix $p\in [1,\infty)$ and $\e\in (0,1)$. Choose an integer $n\in \N$
such that
\begin{equation}\label{eq:p range}
n\le p<n+1,
\end{equation}
and define
\begin{equation}\label{eq:def alpha}
\alpha\eqdef \frac{2n+2}{p}-1+\e.
\end{equation}
Note that by the choice of $n$ we have $\alpha\in
[1+\e,3+\e)\subseteq (1,4)$. For every $x,z\in \R^{2n+1}\setminus
\{0\}$ define $T(x)(y)\in \R$ by
\begin{equation}\label{eq:def T}
T(x)(z)\eqdef \frac{1}{N(x^{-1}z)^{\alpha}}-\frac{1}{N(z)^{\alpha}}.
\end{equation}

\begin{lemma}\label{lem:int on ball}
For every $R\in (0,\infty)$ we have
$$
\left(\int_{B_N(0,R)} \frac{dz}{N(z)^{\alpha p}}\right)^{1/p}\asymp \frac{R^{1-\e}}{p(1-\e)^{1/p}}.
$$
\end{lemma}

\begin{proof} This is a straightforward computation. First, since
\begin{equation}\label{eq:box upper inclusion}
B_N(0,R)\subset\left\{z\in \R^{2n+1}:\ \|\pi(z)\|_2\le R\ \wedge \ |z_{2n+1}|\le R^2\right\},
\end{equation}
 by integration in polar coordinates on $\R^{2n}$ we have,
\begin{align}
\nonumber&\left(\int_{B_N(0,R)} \frac{dz}{N(z)^{\alpha p}}\right)^{1/p}\\&
\le
\left(\int_{0}^R 2nv_{2n}r^{2n-1} \left(\int_{-R^2}^{R^2}
\frac{dt}{\left(r^4+t^2\right)^{\alpha p/4}}\right)dr\right)^{1/p} \label{eq:vn}\\
&\lesssim \frac{1}{p}\left(\int_0^R \frac{dr}{r^{\alpha
p-2n-1}}\right)^{1/p}\asymp \frac{R^{1-\e}}{p(1-\e)^{1/p}}, \label{eq:upper on ball}
\end{align}
where in~\eqref{eq:vn} $v_{2n}=\pi^n/n!$ denotes the volume
of the Euclidean unit ball in $\R^{2n}$. The penultimate inequality
of~\eqref{eq:upper on ball} uses the fact that $v_{2n}^{1/p}\asymp 1/p$ (recall~\eqref{eq:p range}). In
the final step of~\eqref{eq:upper on ball} we used~\eqref{eq:def
alpha}. Using the inclusion
\begin{equation}\label{eq:box lower inclusion}
B_N(0,R)\supset\left\{z\in \R^{2n+1}:\ \|\pi(z)\|_2\le \frac{R}{\sqrt[4]{2}}\ \wedge \ |z_{2n+1}|\le \frac{R^2}{\sqrt{2}}\right\}
\end{equation}
in place of~\eqref{eq:box upper inclusion}, the reverse inequality
is proved analogously.
\end{proof}

\begin{corollary}\label{cor:on ball}
For every $x\in \R^{2n+1}$ and $K\in [1/3,\infty)$ we have
\begin{equation}\label{eq:cor on balls}
\frac{N(x)^{1-\e}}{p(1-\e)^{1/p}}\lesssim  \left(\int_{B_N(0,KN(x))} |T(x)(z)|^pdz\right)^{1/p}
\lesssim \frac{K^{1-\e}N(x)^{1-\e}}{p(1-\e)^{1/p}}.
\end{equation}
\end{corollary}
\begin{proof}
For the upper bound note that by the definition of $T$
in~\eqref{eq:def T},
\begin{align*}
&\left(\int_{B_N(0,KN(x))} |T(x)(z)|^pdz\right)^{1/p}\\&\le \left(\int_{B_N(0,KN(x))}
\frac{dz}{N(z)^{\alpha p}}\right)^{1/p}+
\left(\int_{B_N(0,KN(x))} \frac{dz}{N(x^{-1}z)^{\alpha p}}\right)^{1/p}\\
&= \left(\int_{B_N(0,KN(x))}
\frac{dz}{N(z)^{\alpha p}}\right)^{1/p}+
\left(\int_{B_N\left(x^{-1},KN(x)\right)} \frac{dw}{N(w)^{\alpha p}}\right)^{1/p},
\end{align*}
where we used the fact that the Lebesgue measure on $\R^{2n+1}$ is a
Haar measure of the Heisenberg group, and the left-invariance of the
metric $d_N$. By the triangle inequality, $B_N(x^{-1},KN(x))\subset
B_N(0,4KN(x))$, so the rightmost inequality in~\eqref{eq:cor on
balls} follows from Lemma~\ref{lem:int on ball}.

If $z\in B_N(0,N(x)/3)$ then $N(x^{-1}z)\ge N(x)-N(z)\ge 2N(z)$,
whence $N(z)^{-\alpha}\ge 2N(x^{-1}z)^{-\alpha}$ (using $\alpha\ge
1$). So $|T(x)(z)|\ge N(z)^{-\alpha}$ for every $z\in
B_N(0,N(x)/3)$. Since $K\ge 1/3$, the leftmost inequality
in~\eqref{eq:cor on balls} now follows from Lemma~\ref{lem:int on
ball}.
\end{proof}

\begin{lemma}\label{lem:p integrability}
For every $x\in \R^{2n+1}$ we have $T(x)\in L_p(\R^{2n+1})$ and
\begin{equation}\label{eq:desired T upper}
\|T(x)\|_{L_p(\R^{2n+1})}\lesssim \frac{N(x)^{1-\e}}{p}
\left(\frac{1}{\e^{1/p}}+\frac{1}{(1-\e)^{1/p}}\right).
\end{equation}
\end{lemma}

\begin{proof} If $z\in \R^{2n+1}\setminus B_N(0,2N(x))$ then by the triangle inequality for
the Kor\'anyi norm we have $N(x^{-1}z)\le N(x)+N(z)\le 2N(z)$ and
$N(x^{-1}z)\ge N(z)-N(x)\ge N(z)/2$. Consequently  $N(x^{-1}z)\asymp
N(z)$ for every $z\in \R^{2n+1}\setminus B_N(0,2N(x))$, and it
therefore follows that
$$
 |T(x)(z)|\lesssim \frac{|N(x^{-1}z)-N(z)|}{N(z)^{\alpha+1}}\le \frac{N(x)}{N(z)^{\alpha+1}}.
$$
We conclude that
\begin{align}\label{eq:UV}
\nonumber&\left(\int_{\R^{2n+1}\setminus B_N(0,2N(x))}|T(x)(z)|^pdz\right)^{1/p}\\&\le
N(x)\left(\int_{\R^{2n+1}\setminus B_N(0,2N(x))}\frac{dz}{N(z)^{(\alpha+1)p}}\right)^{1/p}\nonumber\\
&\le N(x) \left(\int_{U_x}\frac{dz}{N(z)^{(\alpha+1)p}}\right)^{1/p}
+ N(x) \left(\int_{V_x}\frac{dz}{N(z)^{(\alpha+1)p}}\right)^{1/p},
\end{align}
where
$$
U_x\eqdef \left\{z\in \R^{2n+1}\setminus B_N\left(0,2N(x)\right):\ \|\pi(z)\|_2\ge \sqrt{|z_{2n+1}|}\right\}
$$
and
$$
V_x \eqdef \left\{z\in \R^{2n+1}\setminus B_N\left(0,2N(x)\right):\ \|\pi(z)\|_2\le \sqrt{|z_{2n+1}|}\right\}.
$$
For $z\in U_x$ we have $\|\pi(z)\|_2\le N(z)\le 2\|\pi(z)\|_2$, and
therefore
\begin{multline}\label{eq:U}
\left(\int_{U_x}\frac{dz}{N(z)^{(\alpha+1)p}}\right)^{1/p}\le
\left(\int_{\{w\in \R^{2n}:\ \|w\|_2\ge N(x)\}} \frac{2dw}
{\|w\|_2^{(\alpha+1)p-2}}\right)^{1/p}\\
= \left(\int_{N(x)}^\infty \frac{4nv_{2n} dr} {r^{(\alpha+1)p-2n-1}}
\right)^{1/p}\lesssim
\frac{1}{p}\left(\int_{N(x)}^\infty\frac{dr}{r^{1+p\e}}\right)^{1/p}\lesssim
\frac{1}{p\e^{1/p}N(x)^\e}.
\end{multline}
If $z\in V_x$ then $\sqrt{|z_{2n+1}|}\le N(z)\le
2\sqrt{|z_{2n+1}|}$, and therefore
\begin{align}\label{eq:V}
\nonumber&\left(\int_{V_x}\frac{dz}{N(z)^{(\alpha+1)p}}\right)^{1/p}\\&\le
\nonumber\left(\int_{N(x)^2}^\infty \mathrm{vol}_{2n}\left(\left\{w\in \R^{2n}:\
\|w\|_2\le \sqrt{t}\right\}\right)\cdot\frac{2dt}
{t^{\frac{(\alpha+1)p}{2}}}\right)^{1/p}\\
&= \left(\int_{N(x)^2}^\infty \frac{2v_{2n}dt}{t^{\frac{(\alpha+1)p}{2}-n}}\right)^{1/p}\lesssim
\frac{1}{p\e^{1/p}N(x)^\e}.
\end{align}
The desired estimate~\eqref{eq:desired T upper} now follows from
substituting~\eqref{eq:U} and~\eqref{eq:V} into~\eqref{eq:UV}, and
using Corollary~\ref{cor:on ball}.
\end{proof}

\begin{corollary}\label{cor:bi lip done}
Define $S:\R^{2n+1}\to L_p(\R^{2n+1})$ by
\begin{equation}\label{eq:defS}
S\eqdef p(1-\e)^{1/p}T.
\end{equation}
Then for every $x,y\in \R^{2n+1}$ we have
$$
d_N(x,y)^{1-\e}\lesssim
\|S(x)-S(y)\|_{L_p(\R^{2n+1})}\lesssim
\left(1+\frac{(1-\e)^{1/p}}{\e^{1/p}}\right)d_N(x,y)^{1-\e}.
$$
\end{corollary}

\begin{proof}
Observe that since the Lebesgue measure is a Haar measure for the
Heisenberg group,
$$\|S(x)-S(y)\|_{L_p(\R^{2n+1})}=p(1-\e)^{1/p}\|T(y^{-1}x)\|_{L_p(\R^{2n+1})}.$$
Hence the desired upper bound on $\|S(x)-S(y)\|_{L_p(\R^{2n+1})}$
follows from Lemma~\ref{lem:p integrability} and the desired lower
bound on $\|S(x)-S(y)\|_{L_p(\R^{2n+1})}$ follows from the leftmost
inequality in Corollary~\ref{cor:on ball}.
\end{proof}

\begin{lemma}\label{lem:image doubling}
Let $S$ be defined as in~\eqref{eq:defS} and let $\f:\R^3\to
\R^{2n+1}$ be the canonical embedding of the corresponding
Heisenberg groups, i.e.,
$$
\forall(a,b,c)\in \R^3,\qquad \f(a,b,c)\eqdef(a,b,\underbrace{0,\ldots,0}_{2n-2\ \mathrm{times}},c)\in \R^{2n+1}.
$$
Then $S\circ\f(\R^{3})\subset L_p(\R^{2n+1})$ is
$2^{8/(1-\e)}$-doubling.
\end{lemma}

\begin{proof}
For $\theta\in (0,\infty)$, recalling~\eqref{eq:def delta theta} we
have
\begin{align}
\nonumber&\|S(\d_\theta(x))-S(\d_\theta(y))\|_{L_p(\R^{2n+1})}
\\\nonumber&=\left(\int_{\R^{2n+1}}\left|\frac{1}{N(\d_\theta(x^{-1})z)^\alpha}
-\frac{1}{N(\d_\theta(y^{-1})z)^\alpha}\right|^pdz\right)^{1/p}\\
&\label{eq:change variable theta}=\left(\int_{\R^{2n+1}}\left|\frac{1}{N(\d_\theta(x^{-1}w))^\alpha}
-\frac{1}{N(\d_\theta(y^{-1}w))^\alpha}\right|^p\theta^{2n+2}dw\right)^{1/p}\\\label{eq:homogebneous}
&= \theta^{\frac{2n+2-\alpha p}{p}}\left(\int_{\R^{2n+1}}\left|\frac{1}{N(x^{-1}w)^\alpha}
-\frac{1}{N(y^{-1}w)^\alpha}\right|^p\theta^{2n+2}dw\right)^{1/p}\\
&= \theta^{1-\e}\|S(x)-S(y)\|_{L_p(\R^{2n+1})}\label{eq:new metric homogeneous},
\end{align}
where~\eqref{eq:change variable theta} uses the change of variable
$z=\d_\theta (w)$ (so $dz=\theta^{2n+2}dw$),
and~\eqref{eq:homogebneous} uses the fact that
$N(\d_\theta(u))=\theta N(u)$ for every $u\in \R^{2n+1}$.
For~\eqref{eq:new metric homogeneous}, recall the definition of
$\alpha$ in~\eqref{eq:def alpha}.

Let $\mu$ be the push-forward of the Lebesgue measure on $\R^3$
under the mapping $S\circ\f$, i.e.,
$$\mu(A)\eqdef
\mathrm{vol}_3
(\f^{-1}(S^{-1}(A)))
$$ for every Borel set $A\subset
S\circ\f(\R^{3})\subset L_p(\R^{2n+1})$. For $f\in L_p(\R^{2n+1})$
and $r\in [0,\infty)$ let $B_p(f,r)$ denote the closed ball of
radius $r$ and center $f$ in $L_p(\R^{2n+1})$, i.e.,
$B_p(f,s)=\{g\in L_p(\R^{2n+1}):\ \|f-g\|_{L_p(\R^{2n+1})}\le r\}$.
By~\eqref{eq:new metric homogeneous} for every $0<r\le R<\infty$ and
every $f\in S\circ\f(\R^3)$ we have
$$
\f^{-1}\circ S^{-1}\left(B_p(f,R)\right)= \f^{-1}(S^{-1}(f))\d_{(R/r)^{1/(1-\e)}}\left(
\f^{-1}\circ S^{-1}\left(B_p(0,r)\right)\right).
$$
Consequently,
$$
\mu\left(B_p(f,R)\right)=\left(\frac{R}{r}\right)^{\frac{4}{1-\e}}
\mathrm{vol}_3\left(\f^{-1}\circ S^{-1}\left(B_p(0,r)\right)\right)>0.
$$
In particular,
$$
\forall\, f\in S\circ\f(\R^{3}),\ \forall\, r\in (0,\infty),\qquad
\frac{\mu\left(B_p(f,2r)\right)}{\mu\left(B_p(f,r)\right)}=2^{4/(1-\e)}.
$$
By a standard packing argument (see e.g.~\cite[page~67]{CW71}), this
implies that $S\circ\f(\R^{3})$ is a $2^{8/(1-\e)}$-doubling subset
of $L_p(\R^{2n+1})$.
\end{proof}

\begin{proof}[Proof of Theorem~\ref{thm:main4}]
The discrete Heisenberg group $\H$ embeds into the continuous
Heisenebrg group $\H_1$ as a co-compact discrete subgroup. Hence
(see e.g.~\cite[Thm.~8.3.19]{BBI01}) the metric space $(\H,d_W)$ is
bi-Lipschitz to a subset of $(\R^3,d_N)$. By taking the mapping
$F_{\e,p}$ of Theorem~\ref{thm:main4} to be the restriction of
$S\circ \phi$ to $\H$, the assertions of Theorem~\ref{thm:main4}
follow from Lemma~\ref{lem:image doubling} and Corollary~\ref{cor:bi
lip done} (observe that $\phi$ is an isometric embedding of
$(\R^3,d_N)$ into $(\R^{2n+1},d_N)$).
\end{proof}

\section{A representation theoretic proof of Theorem~\ref{thm:main4}}

Here we present a different proof of Theorem~\ref{thm:main4} which uses the Schr\"odinger representation of the Heisenberg group $\H_n$. In what follows it will be notationally convenient to consider the Heisenberg group $\H_n$ as being $\R^n\times \R^n\times \R$, equipped with the group product given by
$$
(u,v,w)(u',v',w')\eqdef \left(u+u',v+v',w+w'-2\langle u,v'\rangle +2\langle v,u'\rangle \right),
$$
for every $(u,v,w),(u',v',w')\in \H_n$. The corresponding Kor\`anyi norm is then given by
$$
\forall(u,v,w)\in \H_n,\qquad N(u,v,w)=\left(\left(\|u\|_2^2+\|v\|_2^2\right)^2+w^2\right)^{1/4},
$$
and for $\theta\in (0,\infty)$ the Heisenberg dilation $\d_\theta:\H_n\to \H_n$ is given by
$$
\forall(u,v,w)\in \H_n,\qquad \d_{\theta}(u,v,w)=\left(\theta u,\theta v,\theta^2 w\right).
$$

The Schr\"odinger representation of $\H_n$ corresponding to $\lambda\in (0,\infty)$ is defined as follows. For every $(u,v,w)\in \H_n$ and $h: \R^n\to \C$ define $\sigma_{\lambda}(u,v,w)h:\R^n\to \C$ by
$$
\forall\, x\in \R^n,\qquad \sigma_{\lambda}(u,v,w)h(x)\eqdef e^{i\lambda\left(w-2\langle u,v\rangle\right)+2i\sqrt{\lambda}\langle v,x\rangle}h\left(x-2\sqrt{\lambda} u\right).
$$
One checks that this defines a unitary representation of $\H_n$ on $L_2(\R^n)$, i.e., that for every $h\in L_2(\R^n)$ we have $\|\sigma(x)h\|_{L_2(\R^n)}=\|h\|_{L_2(\R^n)}$  and $\sigma_\lambda(xy)h=\sigma_\lambda(x)\sigma_\lambda(y)h$ for every $x,y\in \H_n$.

Define $g:\R^n\to \R$ to be
$$
g(x)\eqdef e^{-\frac12\|x\|_2^2},
$$
so that
$
\|g\|_{L_2(\R^n)}^2=\int_{\R^n} e^{-\|x\|_2^2}dx=\pi^{n/2}.
$
Then for every $(u,v,w)\in \H_n$ and $x\in \R^n$ we have
$$
\sigma_\lambda(u,v,w)g(x)=e^{i\lambda\left(w-2\langle u,v\rangle\right)+2i\sqrt{\lambda}\langle v,x\rangle-\frac12\|x\|_2^2+2\sqrt{\lambda}\langle u,x\rangle -2\lambda\|u\|_2^2}.
$$
Consequently,
\begin{align}\label{eq:cmputation for each lambda}
\nonumber&\left\|g-\sigma_\lambda(u,v,w)g\right\|_{L_2(\R^n)}^2\\ \nonumber&=2\|g\|_{L_2(\R^n)}^2-2\Re\left(\int_{\R^n} g(x)\sigma_\lambda(u,v,w)g(x)dx\right)\\ \nonumber
&=2\pi^{n/2}-2\Re\left(\int_{\R^n}e^{i\lambda\left(w-2\langle u,v\rangle\right)+2i\sqrt{\lambda}\langle v,x\rangle-\|x\|_2^2+2\sqrt{\lambda}\langle u,x\rangle -2\lambda\|u\|_2^2} dx\right)\\ \nonumber
&= 2\pi^{n/2}-2\Re\left(e^{i\lambda w-\lambda \left(\|u\|_2^2+\|v\|_2^2\right)}\int_{\R^n}e^{-\|x-i\sqrt{\lambda}v-\sqrt{\lambda}u\|_2^2}dx\right)\\
&= 2\pi^{n/2}\left(1-e^{-\lambda \left(\|u\|_2^2+\|v\|_2^2\right)}\cos(\lambda w)\right).
\end{align}

Let  $L_p((0,\infty),L_2(\R^n))$ denote as usual the space of all measurable mappings $F:(0,\infty)\to L_2(\R^n)$ that satisfy
$$
\|F\|_{L_p((0,\infty),L_2(\R^n))}\eqdef \left(\int_0^\infty \|F(\lambda)\|_{L_2(\R^n)}^p\right)^{1/p}<\infty.
$$
Note that since $L_2(\R^n)$ embeds isometrically into $L_p$ (see e.g.~\cite{Woj91}), also $L_p((0,\infty),L_2(\R^n))$ embeds isometrically into $L_p$. For every measurable mapping $F:(0,\infty)\to L_2(\R^n)$ and every $(u,v,w)\in \H_n$ define $\sigma(u,v,w)F:(0,\infty)\to L_2(\R^n)$ by
$$
\forall\, \lambda\in (0,\infty),\qquad \sigma(u,v,w)F(\lambda)\eqdef\sigma_\lambda(u,v,w)F.
$$
Thus $\sigma$ is an action of $\H_n$ on $L_p((0,\infty),L_2(\R^n))$ by isometries. Next, define $G:(0,\infty)\to L_2(\R^n)$ by
\begin{equation}\label{eq:def representation G}
\forall\, \lambda\in (0,\infty),\qquad G(\lambda)\eqdef \frac{g}{\sqrt{2}\pi^{\frac{n}{4}}\cdot \lambda^{\frac{1}{p}+\frac{1-\e}{2}}}.
\end{equation}

\begin{lemma}\label{lem:compute cocycle}
 We have $G-\sigma(u,v,w) G\in L_p((0,\infty),L_2(\R^n))$ for every $(u,v,w)\in \H_n$. Moreover,
\begin{align}
\nonumber&\left\|G-\sigma(u,v,w) G\right\|_{L_p((0,\infty),L_2(\R^n))}\\\label{eq:exact identity G}&=\left(\int_0^\infty \left(1-e^{-\lambda \left(\|u\|_2^2+\|v\|_2^2\right)}\cos(\lambda w)\right)^{p/2}\frac{d\lambda}{\lambda^{1+(1-\e)p/2}} \right)^{1/p}\\
\label{eq:approx identity G}&\asymp \left(\frac{1}{\e^{1/p}}+\frac{1}{(1-\e)^{1/p}}\right)\left(\|u\|_2^2+\|v\|_2^2\right)^{(1-\e)/2}+\frac{|w|^{(1-\e)/2}}{(1-\e)^{1/p}}.
\end{align}
\end{lemma}

\begin{proof}
The identity~\eqref{eq:exact identity G}  is  a substitution of~\eqref{eq:def representation G} into~\eqref{eq:cmputation for each lambda}. Note that
\begin{equation}\label{eq:ab less than 1}
(a,b)\in (0,1)\times (-1,1)\implies \frac13 \le \frac{1-ab}{(1-a)+(1-b)}\le 1.
\end{equation}
 Indeed, the leftmost inequality in~\eqref{eq:ab less than 1} is equivalent to the inequality $(1+b)(1-a)+2a(1-b)\ge 0$, and the rightmost  inequality in~\eqref{eq:ab less than 1} is equivalent to the inequality $(1-a)(1-b)\ge 0$. It follows from~\eqref{eq:ab less than 1} that for every $\lambda\in (0,\infty)$ we have
\begin{equation*}
1-e^{-\lambda \left(\|u\|_2^2+\|v\|_2^2\right)}\cos(\lambda w)\asymp \left(1-e^{-\lambda \left(\|u\|_2^2+\|v\|_2^2\right)}\right)+ \left(1-\cos(\lambda w)\right).
\end{equation*}
Hence,
\begin{align}\label{eq:closed form new embedding}
\nonumber&\left(\int_0^\infty \left(1-e^{-\lambda \left(\|u\|_2^2+\|v\|_2^2\right)}\cos(\lambda w)\right)^{p/2}\frac{d\lambda}{\lambda^{1+(1-\e)p/2}} \right)^{1/p}\\
\nonumber& \asymp \left(\int_0^\infty \left(1-e^{-\lambda \left(\|u\|_2^2+\|v\|_2^2\right)}\right)^{p/2}\frac{d\lambda}{\lambda^{1+(1-\e)p/2}} \right)^{1/p}\\\nonumber&\quad\quad+
\left(\int_0^\infty \left(1-\cos(\lambda w)\right)^{p/2}\frac{d\lambda}{\lambda^{1+(1-\e)p/2}} \right)^{1/p}\\
\nonumber&= \left(\|u\|_2^2+\|v\|_2^2\right)^{(1-\e)/2}\left(\int_0^\infty \frac{\left(1-e^{-t}\right)^{p/2}}{t^{1+(1-\e)p/2}}dt\right)^{1/p}\\&\quad\quad +|w|^{(1-\e)/2}\left(\int_0^\infty \frac{\left(1-\cos t\right)^{p/2}}{t^{1+(1-\e)p/2}}dt\right)^{1/p}.
\end{align}

Note that
\begin{align}\label{eq:first t integral}
\nonumber \left(\int_0^\infty \frac{\left(1-e^{-t}\right)^{p/2}}{t^{1+(1-\e)p/2}}dt\right)^{1/p}&\asymp \left(\int_0^1 \frac{dt}{t^{1-\e p/2}}+\int_1^\infty \frac{dt}{t^{1+(1-\e)p/2}}\right)^{1/p}\\
&\asymp \frac{1}{\e^{1/p}}+\frac{1}{(1-\e)^{1/p}},
\end{align}
and
\begin{align}\label{eq:second t integral}
\nonumber \left(\int_0^\infty \frac{\left(1-\cos t\right)^{p/2}}{t^{1+(1-\e)p/2}}dt\right)^{1/p}&\asymp \left(\int_0^1 t^{(1+\e)p/2-1}dt+\int_1^\infty \frac{dt}{t^{1+(1-\e)p/2}}\right)^{1/p}\\
&\asymp \frac{1}{(1-\e)^{1/p}}
\end{align}

By combining~\eqref{eq:closed form new embedding} with~\eqref{eq:first t integral} and~\eqref{eq:second t integral} we obtain~\eqref{eq:approx identity G}.
\end{proof}

\begin{proof}[Second proof of Theorem~\ref{thm:main4}]
Fix an arbitrary isometric embedding $J$ of $L_p((0,\infty),L_2(\R^n))$ into $L_p$ and
define $Q:H_n\to L_p$ by
$$
\forall, x\in \H_n,\qquad Q(x)\eqdef (1-\e)^{1/p}J\left(G-\sigma(x)G\right).
$$
By Lemma~\ref{lem:compute cocycle}, for every $\theta\in (0,\infty)$ and every $(u,v,w)\in \H_n$,
\begin{align*}
\|Q(\d_\theta(u,v,w))\|_p^p&=\int_0^\infty \frac{\left(1-e^{-\lambda \left(\theta^2\|u\|_2^2+\theta^2\|v\|_2^2\right)}\cos(\lambda \theta^2 w)\right)^{p/2}}{\lambda^{1+(1-\e)p/2}} d\lambda \\
&= \theta^{(1-\e)p}\int_0^\infty \frac{\left(1-e^{-s \left(\|u\|_2^2+\|v\|_2^2\right)}\cos(\lambda  s)\right)^{p/2}}{s^{1+(1-\e)p/2}} ds\\&= \theta^{(1-\e)p} \|Q(u,v,w)\|_p^p.
\end{align*}
Since $\sigma$ is an action of $\H_n$ on $L_p((0,\infty),L_2(\R^n))$ by isometries, for every $x,y\in \H_n$ we have $\|Q(x)-Q(y)\|_p=\|Q(x^{-1}y)\|_p$, and it therefore follows that $\|Q(\d_\theta(x))-Q(\d_\theta(y))\|_p=\theta^{1-\e}\|Q(x)-Q(y)\|_p$ for every $x,y\in \H_n$ and $\theta\in (0,\infty)$. Arguing exactly as in the proof of Lemma~\ref{lem:image doubling}, it follows that $Q(\H_1)\subseteq L_p$ is a $2^{8/(1-\e)}$-doubling subset of $L_p$. It remains to note that by Lemma~\ref{lem:compute cocycle}, for every $x,y\in \H_n$ we have
\begin{equation*}
d_N(x,y)^{1-\e}\lesssim \|Q(x)-Q(y)\|_p\lesssim \left(1+\frac{(1-\e)^{1/p}}{\e^{1/p}}\right)d_N(x,y)^{1-\e}.\qedhere
\end{equation*}
\end{proof}

\subsection*{Acknowledgements}
 V.~L. was supported by ANR grant KInd. A.~N. was supported by NSF
grant CCF-0832795, BSF grant 2010021, the Packard Foundation and the
Simons Foundation. Part of this work was completed while A.~N. was a
Visiting Fellow at Princeton University.


\bibliographystyle{alphaabbrvprelim}
\bibliography{Lp-doubling}

\end{document}